\documentclass[11pt]{amsart}

\usepackage{amsmath, amsthm, amssymb, tikz, hyperref, enumerate,array}
\usepackage{amsfonts,bm}
\usepackage{fullpage}
\usepackage[enableskew,vcentermath]{youngtab}

\usepackage{tikz-cd}
\usetikzlibrary{decorations.pathreplacing,patterns}

\usepackage{ytableau}
\ytableausetup{notabloids}
\ytableausetup{centertableaux}

\theoremstyle{plain}

\newtheorem{theorem}{Theorem}[section]
\newtheorem{proposition}[theorem]{Proposition}

\newtheorem{lemma}[theorem]{Lemma}
\newtheorem{corollary}[theorem]{Corollary}

\theoremstyle{definition}
\newtheorem{definition}[theorem]{Definition}

\newtheorem{example}[theorem]{Example}

\newtheorem{question}[theorem]{Question}
\newtheorem{remark}[theorem]{Remark}
\newtheorem{observation}[theorem]{Observation}

\numberwithin{equation}{section}

\newcommand{\todo}[1]{\vspace{5 mm}\par \noindent
	\marginpar{\textsc{ToDo}} \framebox{\begin{minipage}[c]{0.95
				\textwidth}
			#1 \end{minipage}}\vspace{5 mm}\par}

\newcommand{\RR}{{\mathbb {R}}}
\newcommand{\ZZ}{{\mathbb {Z}}}

\newcommand{\sign}{{\operatorname{sign}}}
\newcommand{\inv}{{\operatorname{inv}}}
\newcommand{\Inv}{{\operatorname{Inv}}}

\newcommand{\odd}{{\operatorname{odd}}}

\newcommand{\blambda}{{\bm{\lambda}}}
\newcommand{\bT}{{\bm{T}}}
\newcommand{\bb}{{\bm{b}}}
\newcommand{\bRT}{{\bm{RT}}}

\newcommand{\Par}{{\operatorname{Par}}}

\title{
	On characters of 
	wreath products}

\author{Ron M. Adin}
\address{Department of Mathematics, Bar-Ilan University, Ramat-Gan 52900, Israel}
\email{radin@math.biu.ac.il}

\author{Yuval Roichman}
\address{Department of Mathematics, Bar-Ilan University, Ramat-Gan 52900, Israel}
\email{yuvalr@math.biu.ac.il}

\date{August 31, 2021; revised: May 29, 2022}

\thanks{Partially supported by the Israel Science Foundation, Grant No.\ 1970/18.}

\begin{document}
	
	\maketitle
	
	\begin{abstract}
		A character identity which relates irreducible character values of the  hyperoctahedral group $B_n$ to those of the symmetric group $S_{2n}$ was recently proved by L\"ubeck and Prasad. 
		Their proof is algebraic 
		and involves Lie theory.
		We present a short combinatorial proof of this identity,
		as well as a
		generalization to other wreath products.
	\end{abstract}

	\section{Introduction}
	
	One of the 
	most important and well-studied finite groups
	is the classical Weyl group of type $B_n$,
	also known as the hyperoctahedral group,
	the group of symmetries of the  hypercube,
	or the group of signed permutations.
	The character theory of the hyperoctahedral group was developed by Specht 
	more than 80 years ago,
	using its presentation as a wreath product.
	While the irreducible characters of the symmetric group $S_n$ are indexed by the integer partitions of $n$, those of $B_n$ are indexed by 
	pairs of partitions of total size $n$,
	or equivalently
	by partitions of $2n$ with an empty 2-core (to be defined below). 
	It is well known that the degree of an irreducible $B_n$-character is equal, up to sign, to the value,
	at the longest element of $S_{2n}$,
	of the irreducible $S_{2n}$-character indexed by the same partition of $2n$; 
	see, e.g., \cite[p.\ 110]{Lusztig}.
	
	\medskip
	
	This phenomenon was recently generalized 
	by L\"ubeck and Prasad~\cite{LPA}, 
	presenting the following character identity,
	which relates the   
	irreducible characters of 
	$B_n$ 
	to those of 
	$S_{2n}$.
	
	\medskip
	
	Recall the notation $\lambda\vdash n$ for an integer partition $\lambda$ of $n$.
	Denote by $\chi^\lambda$ (respectively, $\psi^{\lambda}$) the irreducible character of $S_n$ (respectively, $B_n$)  indexed by $\lambda \vdash n$ (respectively, by $\lambda \vdash 2n$ with an empty $2$-core).
	For a partition $\mu \vdash n$, denote by $\chi^\lambda_\mu$ (respectively, $\psi^\lambda_{(\mu,\varnothing)}$) 
	the evaluation of this character at a conjugacy class 
	of type $\mu$ (respectively, $(\mu,\varnothing)$).
	Denote by $\Par_2(2n)$ the set of all partitions of $2n$ with an empty 2-core. 
	
	\begin{theorem}\label{thm:main}\cite[Theorem 6.1]{LPA}
		There exists a function $\epsilon : \Par_2(2n) \to \{1,-1\}$ such that, 
		for every $\lambda\in \Par_2(2n)$ and $\mu\vdash n$,
		\[
		\psi^{\lambda}_{(\mu,\varnothing)} 
		= \epsilon(\lambda) \chi^\lambda_{2\mu},
		\]
		where $2(\mu_1,\dots,\mu_t):=(2\mu_1,\dots,2\mu_t)$.
	\end{theorem}
	
	The proof in~\cite{LPA} is algebraic in nature, 
	and involves 
	Lie theory. We present here a short combinatorial proof, applying the Murnaghan-Nakayama rule.
	We state it, more generally, for the wreath product $G\wr S_n$ where $G$ 
	is any finite abelian group; 
	see Theorems~\ref{thm:main_r} and~\ref{thm:main_r_abelian} below.

	The rest of the paper is organized as follows. 
	Relevant background and notation are given in Section~\ref{sec:MN}. 
	The main result (Theorem~\ref{thm:main_r}) is stated and proved in Section~\ref{sec:main}. This result is further generalized 
	in Section~\ref{sec:finite_abelian_group} (see Theorem ~\ref{thm:main_r_abelian}).
	Section~\ref{sec:additional} contains some  
	alternative descriptions of
	the $r$-sign function.

	\section{The Murnaghan-Nakayama rule for wreath products}\label{sec:MN}
	
	
	In this section we recall some useful facts from combinatorial character theory, 
	regarding the Murnaghan-Nakayama rule for the symmetric group $S_n$ and for the wreath products $\ZZ_r \wr S_n$.
	
	The Murnaghan-Nakayama rule is an explicit formula for computing values of irreducible characters of the symmetric group; 
	see, e.g., \cite[\S 4.10]{Sagan}. 
	A generalization  to wreath products $G\wr S_n$, where $G$ is any finite group, was described by Stembridge~\cite[Theorem 4.3]{Stembridge}. 
	%
	We now give a very short exposition of the Murnaghan-Nakayama rule for $\ZZ_r \wr S_n$, following \cite[Proposition 2.2]{APR}. 
	We use the term {\em ribbon} instead of the older, equivalent, terms {\em border strip}, {\em skew hook}, and {\em rim hook}.

	\subsection{General version}

	A {\em composition} of a non-negative integer $n$ is a sequence $\lambda = (\ell_1, \ldots, \ell_k)$ of positive integers whose sum is $n$;
	we say that $n$ is the {\em size} of $\lambda$ and $k$ is its {\em length}.
	It is a {\em partition} of $n$ if $\ell_1 \ge \ldots \ge \ell_k$;
	in that case we write $\lambda \vdash n$.
	The only composition (or partition) of $0$ is the empty one, with $k = 0$. 
	The {\em diagram} corresponding to $\lambda$,
	according to the English convention,
	is an array of cells in the plane, arranged in left-justified rows of lengths $\ell_1, \ldots, \ell_k$, from top to bottom.
	
	An {\em $r$-partite partition} of $n$ is an $r$-tuple $\blambda = (\lambda_0, \ldots, \lambda_{r-1})$ such that each $\lambda_i$ is a partition of a non-negative integer $n_i$ and $n_0 + \ldots + n_{r-1} = n$.
	(We shall use boldface to denote $r$-partite
	concepts.)
	An {\em $r$-partite ribbon tableau} of shape $\blambda$ is a sequence
	\[
	\bT: \quad
	\varnothing = \blambda^{(0)} \subseteq \ldots \subseteq \blambda^{(t)} = \blambda
	\]
	of $r$-partite partitions (diagrams) such that each consecutive difference $\bb_i := \blambda^{(i)} \setminus \blambda^{(i-1)}$ $(1 \le i \le t)$, as an $r$-tuple of skew shapes, has $r-1$ empty parts and one nonempty part which is a {\em ribbon}, namely a connected skew shape ``of width $1$'';
	explicitly, a ribbon is a sequence of cells in which consecutive cells share an edge, and the steps are either due East or due North.
	For each $1 \le i \le t$,
	let $f_{\bT}(i) \in \{0, \ldots, r-1\}$ be the index of the nonempty part in the $r$-tuple $\bb_i$,
	let $\ell_{\bT}(i)$ be the {\em length} (number of cells) of this part,
	and let $ht_{\bT}(i)$ be its {\em height} (one less than its number of rows).
	An $r$-partite ribbon tableau can also be described by 
	an $r$-tuple of tableaux, in which the cells of each ribbon $\bb_i$ are marked $i$ $(1 \le i \le t)$.
	
	
	\begin{example}\label{ex:ribbon_tableau}
		Here is a $3$-partite ribbon tableau $\bT$ of shape $\blambda = ((4,3),(2),(1,1))$, with $t = 4$. 
		The ribbon indices (omitting the subscript $\bT$) are $f(1) = f(3) = 0$, $f(4) = 1$, $f(2) = 2$,
		with corresponding lengths $\ell(1) = 3$, $\ell(3) = 4$, $\ell(4) = \ell(2) = 2$,
		and heights $ht(1) = ht(3) = ht(2) = 1$, $ht(4) = 0$:
		\[
		\left(
		\,\,
		\begin{ytableau}
		1 & 1 & 3 & 3 \\
		1 & 3 & 3 \\
		\end{ytableau}
		\quad,\quad
		\begin{ytableau}
		4 & 4 \\
		\end{ytableau}
		\quad,\quad
		\begin{ytableau}
		2 \\
		2 \\
		\end{ytableau}
		\,\,
		\right) \, .
		\]
	\end{example}
	
	
	
	\medskip
	
	The {\em wreath product} $\ZZ_r\wr S_n$ is the semidirect product of 
	$\ZZ_r^n$, the $n$-th direct power of the cyclic group
	$\ZZ_r$,
	with the symmetric group $S_n$, obtained by the natural $S_n$-action on the $n$ copies of $\ZZ_r$, namely 
	\[
	\ZZ_r\wr S_n := \{(\sigma, (z_1,\ldots,z_n)) \,:\, \sigma \in S_n,\, z_i \in\ZZ_r\,(\forall i)\}
	\]
	with the group operation
	\[
	(\sigma, (z_1,\ldots,z_n)) \cdot (\tau, (y_1,\ldots,y_n)) :=
	(\sigma\tau, (z_{\tau^{-1}(1)} + y_1, \ldots, z_{\tau^{-1}(n)} + y_n)).
	\]
	$\ZZ_r\wr S_n$ can also be viewed as a group of $r$-colored permutations, consisting of all the permutations of the set of $rn$ colored digits
	$\{(i,z) \,:\, 1\le i\le n,\, z \in \ZZ_r\}$ which are $\ZZ_r$-equivariant, in the sense that if 
	$\pi(i,z) = (j,y)$ then $\pi(i,z+x) = (j,y+x)$ for all $x \in \ZZ_r$.
	The {\em cycle decomposition} of an element $(\sigma, (z_1,\ldots,z_n)) \in \ZZ_r\wr S_n$ is the decomposition of its underlying permutation $\sigma \in S_n$ as a product of disjoint cycles, 
	with each cycle $c = (i_1, \ldots, i_k)$ assigned a corresponding {\em color} $z(c) := z_{i_1} + \ldots + z_{i_k} \in \ZZ_r$.
	The corresponding {\em cycle structure} is 
	the $r$-partite partition $\blambda = (\lambda_0, \ldots, \lambda_{r-1})$,
	where each partition $\lambda_j$ $(0 \le j \le r-1)$ records the cycle lengths of color $j$. 
	The conjugacy classes of $\ZZ_r \wr S_n$,
	as well as its irreducible characters,
	are indexed by the $r$-partite partitions of $n$.
	
	\begin{theorem}\label{thm:MN}{\rm [Murnaghan-Nakayama rule for $\ZZ_r\wr S_n$]}
		Fix an arbitrary ordering $c = (c_1, \ldots, c_t)$ of the disjoint cycles of an element $\pi \in \ZZ_r \wr S_n$. Let $\ell(c_i)$ be the length of the cycle $c_i$, and let $z(c_i) \in \ZZ_r$ be its color.
		Then, for any $r$-partite partition $\blambda$ of $n$,
		\[
		\psi^\blambda(\pi) = \sum_{\bT \in \bRT_c(\blambda)}
		\prod_{i=1}^{t} (-1)^{ht_{\bT}(i)} \omega^{f_{\bT}(i) \cdot z(c_i)},
		\]
		where $\bRT_c(\blambda)$ is the set of all $r$-partite ribbon tableaux $\bT$ of shape $\blambda$ such that $\ell_{\bT}(i) = \ell(c_i)$ $(\forall i)$;
		$f_{\bT}(i) \in \ZZ_r$ and $ht_{\bT}(i)$ are, respectively, the $i$-th index and height of $\bT$, as above;
		and $\omega := e^{2\pi i / r}$.
	\end{theorem}
	
	For $r = 1$ this reduces to the usual Murnaghan-Nakayama rule for $S_n$:
	\[
	\chi^\lambda(\sigma) = \sum_{T \in RT_c(\lambda)} \prod_{i=1}^{t} (-1)^{ht_{T}(i)},
	\]
	where $c = (c_1, \ldots, c_t)$ is an arbitrary ordering of the cycles of a permutation $\sigma \in S_n$,
	$\lambda$ is a partition of $n$, 
	and $RT_c(\lambda)$ is the set of all ribbon tableaux $T$ of shape $\lambda$ such that $\ell_{T}(i) = \ell(c_i)$ $(\forall i)$.
	
	\bigskip

	\subsection{A restatement}\label{sec:peelings}

	We want to restate Theorem~\ref{thm:MN},
	in the special case where $z(c_i) = 0$ $(\forall i)$,
	with the following notational changes:
	\begin{enumerate}
		\item
		Use a 0/1 encoding of partitions.
		\item 
		Use a recursive interpretation of ($r$-partite) ribbon tableaux.
		\item
		Replace each $r$-partite partition by a single partition.
	\end{enumerate}
	
	Let $\lambda$ be a partition, and let $D = [\lambda]$ be the corresponding diagram, drawn according to the English convention, so that row lengths weakly decrease from top to bottom.
	The {\em boundary sequence} of $\lambda$ is 
	a finite $0/1$ sequence $\partial(\lambda) = (\delta_1, \ldots, \delta_t)$,
	constructed as follows: 
	start at the southwestern corner of the diagram $D$, 
	and proceed along the edges of the southeastern boundary 
	up to the northeastern corner;
	encode each horizontal (east-bound) step by $1$, and each vertical (north-bound) step by $0$. 
	Thus $\partial(\lambda)$ starts with a $1$ and ends with a $0$ (unless $\lambda$ is the empty partition, for which $\partial(\lambda)$ is the empty sequence). 
	Each $1$ corresponds to a column of $D$ (columns ordered from left to right),
	and each $0$ corresponds to a row of $D$ (rows ordered from bottom to top). 
	
	\begin{observation}\label{obs:peeling}
		For $[\mu]\subset [\lambda]$, 
		the skew diagram $[\lambda/\mu] = [\lambda]\setminus [\mu]$ is a ribbon of length $k$ 
		if and only if $\partial(\mu)$ is obtained from $\partial(\lambda)$
		by exchanging two entries $\delta_j=1$ and $\delta_{j+k}=0$ for some $j$
		(and deleting leading $0$-s and trailing $1$-s from the resulting sequence).
	\end{observation}
	
	In the situation described in Observation~\ref{obs:peeling} we say that $\mu$ is obtained from $\lambda$ by {\em peeling} a ribbon of length $k$.
	Thus a ribbon tableau of shape $\lambda$ corresponds to a sequence of peelings of ribbons from $\lambda$,
	and similarly for $r$-partite ribbon tableaux.
	
	\begin{example}\label{ex:peeling}
		The $3$-partite ribbon tableau in Example~\ref{ex:ribbon_tableau} corresponds to the following sequence of ribbon peelings,
		where cells are labeled for clarity:
		\[
		\left( \,\,
		\begin{ytableau}
		1 & 1 & 3 & 3 \\
		1 & 3 & 3 \\
		\end{ytableau}
		\quad,\quad
		\begin{ytableau}
		4 & 4 \\
		\end{ytableau}
		\quad,\quad
		\begin{ytableau}
		2 \\
		2 \\
		\end{ytableau}
		\,\, \right)
		\quad
		\longrightarrow
		\quad
		\left( \,\,
		\begin{ytableau}
		1 & 1 & 3 & 3 \\
		1 & 3 & 3 \\
		\end{ytableau}
		\quad,\quad
		\varnothing
		\quad,\quad
		\begin{ytableau}
		2 \\
		2 \\
		\end{ytableau}
		\,\, \right) 
		\quad
		\longrightarrow
		\]
		
		\[
		\left( \,\,
		\begin{ytableau}
		1 & 1 \\
		1 \\
		\end{ytableau}
		\quad,\quad
		\varnothing
		\quad,\quad
		\begin{ytableau}
		2 \\
		2 \\
		\end{ytableau}
		\,\, \right)
		\quad
		\longrightarrow
		\quad
		\left( \,\,
		\begin{ytableau}
		1 & 1 \\
		1 \\
		\end{ytableau}
		\quad,\quad
		\varnothing
		\quad,\quad
		\varnothing
		\,\, \right)
		\quad
		\longrightarrow
		\quad
		\left( \,\,
		\varnothing
		\quad,\quad
		\varnothing
		\quad,\quad
		\varnothing
		\,\, \right)
		\, .
		\]
		It also corresponds to the following sequence of $3$-tuples of boundary sequences, where the exchanged entries are marked:
		\[
		(111010, \check{1}1\check{0}, 100)
		\,
		\longrightarrow
		\,
		(1\check{1}101\check{0}, \varnothing, 100)
		\,
		\longrightarrow
		\,
		(1010, \varnothing, \check{1}0\check{0})
		\,
		\longrightarrow
		\,
		(\check{1}01\check{0}, \varnothing, \varnothing)
		\,
		\longrightarrow
		\,
		(\varnothing, \varnothing, \varnothing).
		\]
	\end{example}
	
	\smallskip
	
	So far we have described two of the three changes in interpretation that we intend to introduce; 
	let us now describe the third.
	
	\begin{definition}
		Let $\delta = (\delta_1, \ldots, \delta_t)$ be a finite 0/1 sequence containing both $0$-s and $1$-s. Define
		\[
		m_i := |\{1 \le j \le i \,:\, \delta_j = 1\}|
		- |\{i+1 \le j \le t \,:\, \delta_j = 0 \}|
		\qquad (1 \le i \le t),
		\]
		namely the number of $1$-s weakly preceding $\delta_i$ minus the number of $0$-s strictly succeeding $\delta_i$.
	\end{definition}
	
	\begin{observation}
		\[
		m_{i+1} - m_i = 1
		\qquad (1 \le i \le t-1)
		\]
		and $m_1 \le 0 < m_t$.
		Therefore there is a unique index $1 \le i \le t-1$ satisfying $m_i = 0$. 
		The position between indices $i$ and $i+1$ is called the {\em anchor} of the sequence;
		the number of $1$-s preceding it is equal to the number of $0$-s succeeding it.
		This position is invariant under addition of leading $0$-s and trailing $1$-s to the sequence.
	\end{observation}
	
	\begin{example} 
		Here is a 0/1 sequence, with its anchor denoted by the separator ``$|$'':
		\[
		\begin{matrix}
		i: & 1 & 2 & | & 3 & 4 & 5 \\
		\delta_i: & 1 & 0 & | & 1 & 1 & 0 \\
		m_i: & -1 & 0 & | & 1 & 2 & 3 \\
		\end{matrix}
		\]
		and here is the same sequence with some leading $0$-s and trailing $1$-s added:
		\[
		\begin{matrix}
		i: & 1 & 2 & 3 & 4 & | & 5 & 6 & 7 & 8 \\
		\delta_i: & 0 & 0 & 1 & 0 & | & 1 & 1 & 0 & 1 \\
		m_i: & -3 & -2 & -1 & 0 & | & 1 & 2 & 3 & 4 \\
		\end{matrix}
		\]
	\end{example}
	
	\begin{definition}\label{def:quotient}
		Let $\Par$ be the set of all partitions of integers, including the empty partition, and let $r$ be a positive integer.
		Define a function $\varphi_r : \Par^r \to \Par$,
		on $r$-tuples of partitions, as follows:
		for $\lambda^{(0)}, \ldots, \lambda^{(r-1)} \in \Par$,
		the partition $\lambda = \varphi_r(\lambda^{(0)}, \ldots, \lambda^{(r-1)})$ is obtained by the following procedure.
		\begin{enumerate}
			\item 
			Consider the $r$ boundary sequences $\partial(\lambda^{(0)}), \ldots, \partial(\lambda^{(r-1)})$.
			\item
			Add to these sequences leading $0$-s and trailing $1$-s such that the resulting sequences $s^{(0)}, \ldots, s^{(r-1)}$ have the same length $t$ and the same position of the anchor.
			\item
			Merge the sequences $s^{(0)}, \ldots, s^{(r-1)}$ into a single sequence $s$ of length $rt$, in an interlacing fashion: 
			\[
			s_1^{(0)}, s_1^{(1)}, \ldots, s_1^{(r-1)}, s_2^{(0)}, s_2^{(1)}, \ldots, s_2^{(r-1)}, \ldots,
			s_t^{(0)}, s_t^{(1)}, \ldots, s_t^{(r-1)}.
			\]
			\item
			Let $\lambda$ be the unique partition such that $\partial(\lambda)$ is equal to $s$, with leading $0$-s and trailing $1$-s removed.
		\end{enumerate}
	\end{definition}
	
	\begin{remark}\label{rem:r-to-1}
		The function $\varphi_r$ is well-defined, namely independent of the precise lengthening of the sequences is step (2). 
		It is injective, but (for $r > 1$) not surjective. Its image, denoted $\Par_r$, consists of all partitions with an {\em empty $r$-core}, namely partitions which can be reduced to the empty partition by some sequence of peelings of ribbons of length $r$.
		For each $\lambda \in \Par_r$, 
		the unique preimage $\varphi_r^{-1}(\lambda) = (\lambda^{(0)}, \ldots, \lambda^{(r-1)})$ 
		is also called the {\em $r$-quotient} of $\lambda$.
		Note that,
		by Observation~\ref{obs:peeling},
		peeling a ribbon of length $r$ from $\lambda$ is equivalent to peeling a ribbon of length $1$, namely a single cell, from one of $\lambda^{(0)}, \ldots, \lambda^{(r-1)}$.
		It follows that $|\lambda| = r \cdot (|\lambda^{(0)}| + \ldots + |\lambda^{(r-1)}|)$ for $\lambda \in \Par_r$.
		For more details see, e.g., \cite[\S 9]{AR}.
		Note that the convention there differs slightly from that of~\cite{LPA}, which follows the abacus interpretation of~\cite[Section 2.7]{JK}; 
		the difference amounts to a cyclic shift of the $r$-quotient.
	\end{remark}
	
	\begin{example}\label{ex:phi}
		For $r = 3$ and the shapes in Example~\ref{ex:peeling},
		\[
		\begin{matrix}
		\lambda^{(0)} = (4,3) & \mapsto & \partial(\lambda^{(0)}) = 11|1010 & \mapsto & s^{(0)} = 11|1010 \\
		\lambda^{(1)} = (2) & \mapsto & \partial(\lambda^{(1)}) = 1|10 & \mapsto & s^{(1)} = 01|1011 \\
		\lambda^{(2)} = (1,1) & \mapsto & \partial(\lambda^{(2)}) = 10|0 & \mapsto & s^{(2)} = 10|0111 \\
		\end{matrix}
		\]
		and therefore $\lambda = \varphi_3(\lambda^{(0)},\lambda^{(1)},\lambda^{(2)})$ is obtained by
		\[
		\begin{matrix}
		s = 101110|110001111011 & \mapsto & \partial(\lambda) = 101110|1100011110 & \mapsto & \lambda = (10,6,6,6,4,1).
		\end{matrix}
		\]
		Indeed,
		$|\lambda^{(0)}| + |\lambda^{(1)}| + |\lambda^{(2)}| = 7 + 2 + 2 = 11$
		and $|\lambda| = 33 = 3 \cdot 11$.
	\end{example}

	\smallskip
	
	We now restate Theorem~\ref{thm:MN} in the special case where the colors of all cycles are zero.
	Note that, by Remark~\ref{rem:r-to-1}, the irreducible characters of $\ZZ_r \wr S_n$ can be indexed by partitions $\lambda$ of $rn$ with an empty $r$-core,
	instead of $r$-partite partitions of $n$.
	
	\begin{theorem}\label{thm:MN_boundary}
		Let $\lambda$ be a partition of $rn$ with an empty $r$-core,
		and let $\mu=(\mu_1,\dots,\mu_t)$ be a composition of $n$.
		The character $\psi^\lambda_{(\mu,\varnothing,\dots,\varnothing)}$ is equal to the sum of values obtained by all possible applications of the following ``peeling algorithm'':
		
		\smallskip
		
		Initialization: $\mu:=(\mu_1,\dots,\mu_t)$, 
		$\delta := \partial(\lambda)$, 
		and $\epsilon := 1$.
		
		\smallskip
		
		Main loop:
		
		\begin{enumerate}
			\item
			If $t = 0$ then end the loop and output $\epsilon$.
			\item 
			Choose an index $q$ such that $\delta_q=1$ and $\delta_{q+r\mu_t}=0$. 
			If there is no such index, set $\epsilon := 0$ and end the loop. [This is the case of an unsuccessful peeling.]
			\item 
			Redefine $\delta$ by switching the two entries, i.e., letting $\delta_q := 0$ and $\delta_{q+r\mu_t} := 1$. 
			\item
			Multiply $\epsilon$ by $-1$ if  the number of zeros in $\delta$ between the switched entries, in positions congruent to $q \pmod r$ only, is odd (and by 1 otherwise). 
			\item
			Redefine $\mu:=(\mu_1,\dots,\mu_{t-1})$, $t := t-1$.
			\item
			Go to step (1).
		\end{enumerate}
	\end{theorem}
	
	\begin{remark}
		There is a choice in step (2) of the algorithm.
		Each successful round of the main loop (ending with $t = 0$)  is called a {\em $\mu$-peeling}, and contributes a summand $\epsilon = \pm 1$ to $\psi^\lambda_{(\mu,\varnothing,\dots,\varnothing)}$.
		%
	\end{remark}
	\begin{proof}
		Let $\pi \in \ZZ_r \wr S_n$ belong to the conjugacy class corresponding to $(\mu,\varnothing,\dots,\varnothing)$.
		Each ordering $(c_1, \ldots, c_t)$ of the cycles of $\pi$ with lengths $\ell(c_i) = \mu_i$ $(\forall i)$ has, by assumption, colors $z(c_i) = 0$ $(\forall i)$.
		Therefore the formula in Theorem~\ref{thm:MN} reduces to
		\[
		\psi^\blambda(\pi) = \sum_{\bT \in \bRT_c(\blambda)} \prod_{i=1}^{t} (-1)^{ht_{\bT}(i)} .
		\]
		Each $r$-partite ribbon tableau  $\bT \in \bRT_c(\blambda)$ corresponds to a sequence of ribbon peelings of the $r$-partite partition $\blambda = (\lambda^{(0)},\dots,\lambda^{(r-1)})$.
		Peeling a ribbon of length $\mu_t$ from $\lambda^{(j)}$ is equivalent to 
		switching two entries $\delta_k = 1$ and $\delta_{k+\mu_t} = 0$ in $\partial(\lambda^{(j)})$. 
		The relevant height $ht_{\bT}(t)$ is the number of zeros in $\partial(\lambda^{(j)})$ strictly between $\delta_k$ and $\delta_{k+\mu_t}$.
		This can be restated in terms of
		$\lambda = \varphi_r(\lambda^{(0)},\dots,\lambda^{(r-1)})$:
		by Definition~\ref{def:quotient},
		peeling the ribbon corresponds to switching two entries $\delta_q = 1$ and $\delta_{q+r\mu_t} = 0$ in $\partial(\lambda)$, for a suitable index $q$.
		The height $ht_{\bT}(t)$ is the number of zeros in $\partial(\lambda)$ strictly between $\delta_q$ and $\delta_{q+\mu_t}$, but only in positions congruent to $q \pmod r$. This explains step $(4)$ of the algorithm.
	\end{proof}
	
	\begin{example}
		The peeling in Example~\ref{ex:peeling}, viewed as a peeling of $\lambda = \varphi_3(\lambda^{(0)},\lambda^{(1)},\lambda^{(2)})$, as in Example~\ref{ex:phi}, is
		\[
		1011\check{1}0|1100\check{0}11110
		\,
		\longrightarrow
		\,
		101\check{1}00|110011111\check{0}
		\,
		\longrightarrow
		\,
		10\check{1}000|11\check{0}0
		\,
		\longrightarrow
		\,
		\check{1}00000|111\check{0}
		\,
		\longrightarrow
		\,
		| = \varnothing
		\]
		The corresponding numbers of zeros, in intermediate positions with the correct remainder $\!\!\pmod 3$, are
		0, 1, 1, and 1. The contribution to the character value is therefore $(-1)^{0+1+1+1} = -1$.
	\end{example}

	\section{Main Theorem}\label{sec:main}
	
	Noting that $B_n\cong \ZZ_2\wr S_n$, we state the following 
	generalization of Theorem~\ref{thm:main}.
	
	
	\begin{theorem}\label{thm:main_r}
		For every positive integer $r$
		there exists a function $\sign_r : \Par_r \to \{1,-1\}$ such that, 
		for every $r$-partite partition $\blambda=(\lambda^{(0)},\ldots,\lambda^{(r-1)})$ of a positive integer $n$
		and every composition $\mu = (\mu_1,\ldots,\mu_t)$ of $n$: 
		\[
		\psi^{(\lambda^{(0)},\ldots,\lambda^{(r-1)})}_{(\mu,\varnothing,\ldots,\varnothing)} = \sign_r(\lambda)\cdot \chi^\lambda_{r\mu},
		\]
		where 
		$\psi^{(\lambda^{(0)},\ldots,\lambda^{(r-1)})}$ is the irreducible $\ZZ_r \wr S_n$-character indexed by $(\lambda^{(0)},\ldots,\lambda^{(r-1)})$,
		$\chi^\lambda$ is the irreducible $S_{rn}$-character indexed by $\lambda := \varphi_r(\lambda^{(0)}, \ldots, \lambda^{(r-1)}) \in \Par_r$,
		and $r\mu := (r\mu_1,\dots,r\mu_t)$.
	\end{theorem}
	
	Let us start with a sequence of observations and definitions, leading to an explicit expression for $\sign_r(\lambda)$ in Definition~\ref{def.inv_r} and Lemma~\ref{t:sign_independent_of_k}.
	Then Proposition~\ref{prop:mainr} will imply Theorem~\ref{thm:main_r}.
	
	%
	%
	
	
	As remarked before Observation~\ref{obs:peeling} above,
	if $\lambda$ is any partition then
	the zeros in the boundary sequence $\partial(\lambda)$ correspond to  
	the parts of $\lambda$, in reverse order;
	equivalently, to the rows of diagram of $\lambda$, ordered from bottom to top.
	In the sequel it will be convenient to fix a positive integer $k$ and consider partitions with {\em at most} $k$ parts, namely $\lambda = (\ell_1, \ldots, \ell_k)$ where $\ell_1 \ge \ldots \ge \ell_k \ge 0$.
	We thus require the boundary sequence $\partial (\lambda)$ to contain exactly $k$ zeros, by allowing leading zeros. 
	
	\begin{observation}\label{obs:a}
		If $\lambda = (\ell_1, \ldots, \ell_k)$ 
		where $\ell_1 \ge \ldots \ge \ell_k \ge 0$,
		then the position in $\partial(\lambda)$ of the zero corresponding to $\ell_i$ $(1 \le i \le k)$ is equal to $\ell_i + k - i + 1$; 
		there are $\ell_i$ ones and $k-i$ zeros preceding it.
	\end{observation}
	
	\begin{definition}\label{def.a_sequence}
		Let $\lambda = (\ell_1, \ldots, \ell_k)$ with $\ell_1 \ge \ldots \ge \ell_k \ge 0$.
		For each $1 \le i \le k$, let $0 \le a_i \le r-1$ be the remainder obtained upon dividing $\ell_i + k - i$ by $r$.
		The (length $k$) {\em row-color sequence} of $\lambda$ is $a^{(k)}(\lambda) := (a_1, \ldots, a_k)$. 
	\end{definition} 
	
	\begin{remark}
		The numbers $\ell_i +k - i$ are called {\em $\beta$-numbers} in~\cite{JK, LPA}.
	\end{remark}
	
	Let $\Par_r(rn)$ denote the set of all partitions of $rn$ with an empty $r$-core.
	
	\begin{lemma}\label{t:permutation}
		If $\lambda = (\ell_1, \ldots, \ell_k)$ has an empty $r$-core and the empty partition $\varnothing$ is represented as a sequence of $k$ zeros,
		then the sequence $a^{(k)}(\lambda)$ is a permutation of the sequence $a^{(k)}(\varnothing)$. 
	\end{lemma}
	
	\begin{proof}
		Assume that $\lambda = (\ell_1, \ldots, \ell_k) \in \Par_r(rn)$.
		By assumption, there is a peeling by ribbons of length $r$ which reduces $\lambda$ to the empty partition.
		It suffices to show that the sequence $a^{(k)}(\lambda)$ is a permutation of the sequence $a^{(k)}(\lambda')$, for any partition $\lambda' \in \Par_r(r(n-1))$ obtained from $\lambda$ by peeling one ribbon of length $r$.
		
		Indeed, assume that $\partial(\lambda')$ is obtained from $\partial(\lambda)$ by switching the entries $\delta_q = 1$ and $\delta_{q+r} = 0$.
		Let $i_1, \ldots, i_2$ $(1 \le i_1 \le i_2 \le k)$ be the indices of the rows in the diagram of $\lambda$ corresponding to the zeros in the interval $\delta_q, \ldots, \delta_{q+r}$.
		Note that the order is reversed;
		in particular, $i_1$ corresponds to $\delta_{q+r} = 0$, while $i_2$ corresponds to the first zero after $\delta_q = 1$.
		The switch moves a zero in $\partial(\lambda)$ from position $q+r$ to position $q$, without moving the other zeros.
		By Observation~\ref{obs:a} and Definition~\ref{def.a_sequence},
		the row-color sequence $(a_1, \ldots, a_k)$ of $\lambda$ and the row-color sequence $(a'_1, \ldots, a'_k)$ of $\lambda'$ are related by
		\[
		a'_i = 
		\begin{cases}
		a_i, & \text{if } i < i_1 \text{ or } i > i_2; \\
		a_{i+1}, & \text{if } i_1 \le i \le i_2 - 1; \\
		a_{i_1}, & \text{if } i = i_2.
		\end{cases}
		\]
		The equality $a'_{i_2} = a_{i_1}$ holds since $q-1$ and $q+r-1$ have the same remainder $\!\!\pmod r$.
		Thus the effect of this peeling step on the row-color sequence is a cyclic shift of the entries $a_{i_1}, \ldots, a_{i_2}$.
		In particular, $a^{(k)}(\lambda)$ is a permutation of $a^{(k)}(\lambda')$.
	\end{proof}
	
	\begin{example}\label{ex:color_sequence}
		Let $r = 3$, $n = 6$, and $\lambda = (5,5,4,3,1) \in \Par_3(3 \cdot 6)$.
		If $k = 5$ then 
		$(\ell_1 + 4, \ldots, \ell_5 + 0) = (9,8,6,4,1)$
		and 
		$a^{(5)}(\lambda) = (0,2,0,1,1)$.
		The corresponding representation of the empty partition $\varnothing = (0,0,0,0,0)$
		has 
		$(\ell_1 + 4, \ldots, \ell_5 + 0) = (4,3,2,1,0)$
		and 
		$a^{(5)}(\varnothing) = (1,0,2,1,0)$.
		Clearly $a^{(5)}(\lambda)$ is a permutation of $a^{(5)}(\varnothing)$.
	\end{example}
	
	\begin{definition}\label{def.inv_r}
		For $k$, $\lambda$ and $a^{(k)}(\lambda) = (a_1, \ldots, a_k)$ as in Definition~\ref{def.a_sequence},
		the {\em $r$-inversion set} of $\lambda$ is
		\[
		\Inv_r^{(k)}(\lambda) := \{ (i,j) \,:\, i < j,\, a_i > a_j \}
		\]
		and its {\em $r$-inversion number} is
		\[
		\inv_r^{(k)}(\lambda) := |\Inv_r^{(k)}(\lambda)|.
		\]
		The {\em $r$-sign} of $\lambda$ is
		\[
		\sign_r^{(k)}(\lambda) := (-1)^{\inv_r^{(k)}(\lambda) - \inv_r^{(k)}(\varnothing)},
		\]
		where $\varnothing$ is the empty partition, represented as a sequence of $k$ zeros.
	\end{definition} 
	
	\begin{observation}\label{t:count_transpositions}
		$\sign_r^{(k)}(\lambda) = -1$ if and only if the length of some (equivalently, each) sequence of transpositions transforming $a^{(k)}(\lambda)$ into $a^{(k)}(\varnothing)$ is odd,
		where only transpositions switching two {\em distinct} values in the sequence are counted.
	\end{observation}
	
	\begin{lemma}\label{t:sign_independent_of_k}
		The number $\sign_r^{(k)}(\lambda)$ is independent of $k$,
		as long as $k$ is larger or equal to the number of (positive) parts of $\lambda$.
		We shall therefore denote it simply by $\sign_r(\lambda)$.
	\end{lemma}
	
	\begin{proof}
		If $a^{(k)}(\lambda) = (a_1, \ldots, a_k)$
		then $a^{(k+1)}(\lambda) = (a_1 + 1, \ldots, a_k + 1, 0)$, where addition is modulo $r$. 
		A similar connection holds between $a^{(k)}(\varnothing)$ and $a^{(k+1)}(\varnothing)$, and the claim thus follows from Observation~\ref{t:count_transpositions}.
	\end{proof}
	
	
	As noted in the proof of Theorem~\ref{thm:MN_boundary},
	if $\lambda = \varphi_r(\lambda^{(0)}, \ldots, \lambda^{(r-1)}) \in \Par_r(rn)$ 
	and $m$ is a positive integer, then
	peeling a ribbon of length $rm$ from $\lambda$ is equivalent to peeling a ribbon of length $m$ from one of $\lambda^{(0)}, \ldots, \lambda^{(r-1)}$.
	It follows that,
	for every composition $\mu = (\mu_1,\ldots,\mu_t)$ of $n$, 
	there is a natural bijection between $r\mu$-peelings of $\lambda$ 
	and $\mu$-peelings of its $r$-quotient $(\lambda^{(0)}, \ldots, \lambda^{(r-1)})$.
	To prove Theorem~\ref{thm:main_r}, it thus suffices to show that the signs of matching peelings under this bijection differ by a $\pm1$ factor 
	which depends only on $r$ and $\lambda$.
	
	\begin{proposition}\label{prop:mainr}
		For any partition $\lambda \in \Par_r(rn)$ 
		and any composition $\mu$ of $n$,
		the sign of any $r\mu$-peeling of $\lambda$ and
		the sign of the corresponding $\mu$-peeling of its $r$-quotient differ by the multiplicative factor $\sign_r(\lambda)$.
	\end{proposition}
	

	\begin{proof}
		By induction on $n$.
		Of course, the claim trivially holds for $n = 0$.
		
		Assume that $n > 0$.
		If there is no $r\mu$-peeling of $\lambda$ then
		there is also no $\mu$-peeling of its $r$-quotient, and the claim holds vacuously.
		We can therefore assume that there exists an $r\mu$-peeling of $\lambda$, and consider one of them.
		We further consider only the last entry $r\mu_t$ of $r\mu$, which corresponds to a ribbon of length $r\mu_t$.
		Let $\lambda' \in \Par_r(r(n - \mu_t))$ be the partition obtained by peeling it from $\lambda$.
		
		By assumption, there exists an index $q$ such that $\partial(\lambda')$ is obtained from $\partial(\lambda)$ by switching the entries 
		$\delta_{q} = 1$ and $\delta_{q + r\mu_t} = 0$. 
		We want to show that the sign contribution of this step to the $r\mu$-peeling of $\lambda$ and the sign contribution of this step to the corresponding $\mu$-peeling of its $r$-quotient differ by the multiplicative factor
		\[
		\sign_r(\lambda) / \sign_r(\lambda').
		\]
		Indeed, by Theorem~\ref{thm:MN_boundary}(4),
		the effect of the switch on the sign of the $\mu$-peeling is multiplication by $(-1)^{n_1}$, where $n_1$ is the number of zeros strictly between the switched letters, counting only positions which are congruent to $q \pmod r$.
		On the other hand, by 
		the same theorem
		with $r$ and $n$ replaced by $1$ and $rn$, respectively,
		the effect of this switch on the sign of the $r\mu$-peeling is multiplication by $(-1)^{n_2}$, where $n_2$ is the total number of zeros strictly between the switched letters. 
		Hence, the effect on the ratio of these two signs is multiplication by $(-1)^{n_3}$, where $n_3 = n_2 - n_1$ is the number of zeros between the switched letters, in positions which are {\em not} congruent to $q \pmod r$.
		
		Now let $i_1, \ldots, i_2$ be the indices of the rows corresponding to the zeros between positions $q$ and $q + r\mu_t$ in $\partial(\lambda)$.
		Note that the order is reversed;
		in particular, $i_1$ corresponds to $\delta_{q+r\mu_t} = 0$, while $i_2$ corresponds to the first zero after $\delta_q = 1$.
		Then
		\[
		n_3 =  |\{i \,:\, i_1 < i \le i_2,\, a_i \ne a_{i_1}\}|.
		\]
		By the same argument as in the proof of Lemma~\ref{t:permutation},
		the switch (or peeling step) amounts to a cyclic shift of the entries $a_{i_1}, \ldots, a_{i_2}$ in the row-color sequence.
		Thus, by Definition~\ref{def.inv_r},
		$\inv_r^{(k)}(\lambda) - \inv_r^{(k)}(\lambda')$
		has the same parity as $n_3$ above.
		Thus
		\[
		(-1)^{n_3} 
		= (-1)^{\inv_r^{(k)}(\lambda) - \inv_r^{(k)}(\lambda')}
		= \sign_r(\lambda) / \sign_r(\lambda'),
		\]
		as required.
		This completes the induction step.
	\end{proof}
	
	As remarked above, Proposition~\ref{prop:mainr} implies Theorem~\ref{thm:main_r}.

	\section{Wreath product with a finite abelian group}%
	\label{sec:finite_abelian_group}
	
	As noted by an anonymous referee, the results stated above actually hold for the wreath product of $S_n$ with an arbitrary finite abelian group $G$, not only a finite cyclic group.
	In order to justify this claim, let us first state Stembridge's extension of the Murnaghan-Nakayama formula in full generality.
	
	
	Let $G$ be a finite group,
	$C_G$ its set of conjugacy classes,
	and $I_G$ its set of irreducible characters. Denote $r := |C_G| = |I_G|$.
	The conjugacy classes of $G \wr S_n$ are indexed by functions $\kappa: C_G \to \Par$ such that the sum of the sizes of all partitions is $n$; and the irreducible characters of $G \wr S_n$ are indexed by functions $\chi: I_G \to \Par$ with a similar restriction.
	Now fix a bijection from $\{0, \ldots, r-1\}$ to $I_G$, so that 
	$I_G = \{\theta_0, \ldots , \theta_{r-1}\}$. Then the irreducible characters $\psi^\blambda$ of $G \wr S_n$ are indexed by $r$-partite partitions $\blambda$ of $n$, as defined in Section~\ref{sec:MN}.
	
	View the wreath product $G \wr S_n$ as the group of $n \times n$ pseudo permutation matrices in which the nonzero entries are chosen from $G$, and let $\pi \in G \wr S_n$. If $c = (i_1, \ldots, i_k)$ is a cycle in the permutation in $S_n$ underlying $\pi$, and $g_1, \ldots, g_k \in G$ are the nonzero entries in rows $i_1, \ldots, i_k$ of the matrix $\pi$, then the product $g_k \cdots g_1 \in G$ is well-defined up to a cyclic shift of the indices, thus up to conjugacy in $G$. Let $z(c) \in C_G$ be the corresponding conjugacy class.
	
	Recall, from Section~\ref{sec:MN}, the notion of an $r$-partite ribbon tableaux $\bT$ and the corresponding functions $\ell_\bT$, $ht_\bT$ and $f_\bT$.
	Stembridge's extension of the Murnaghan-Nakayama formula can be stated as follows.
	
	\begin{theorem}\label{thm:MN_G}{\rm \cite[Theorem 4.3]{Stembridge}}
		Let $G$ be a finite group with $|C_G| = |I_G| = r$, specifically $I_G = \{\theta_0, \ldots, \theta_{r-1}\}$, and let $\pi \in G \wr S_n$.
		Fix an arbitrary ordering $c = (c_1, \ldots, c_t)$ of the disjoint cycles of (the permutation underlying)  $\pi$. Let $\ell(c_i)$ be the length of the cycle $c_i$, and let $z(c_i) \in C_G$ be the corresponding conjugacy class, as above. 
		Then, for any $r$-partite partition $\blambda$ of $n$,
		\[
		\psi^\blambda(\pi) = \sum_{\bT \in \bRT_c(\blambda)}
		\prod_{i=1}^{t} (-1)^{ht_{\bT}(i)} \theta_{f_{\bT}(i)}(z(c_i)),
		\]
		where $\bRT_c(\blambda)$ is the set of all $r$-partite ribbon tableau $\bT$ of shape $\blambda$ such that $\ell_{\bT}(i) = \ell(c_i)$ $(\forall i)$;
		$ht_{\bT}(i) \ge 0$ is the $i$-th height of $\bT$; and
		$f_{\bT}(i) \in \{0, \ldots, r-1\}$ is the $i$-th index of $\bT$, as in Theorem~\ref{thm:MN}.
	\end{theorem}
	
	Now assume, further, that $G$ is commutative, so that $r = |C_G| = |I_G| = |G|$ and all irreducible characters are one-dimensional.
	Labeling the elements of the group $G = \{id_G = g_0, g_1, \ldots, g_{r-1}\}$, an element of $G \wr S_n$ with all cycles $c_i$ satisfying $z(c_i) = \{id_G\}$ is of type $(\mu,\varnothing,\ldots,\varnothing)$ for some partition $\mu$ of $n$,
	and then $\theta_j(z(c_i)) = 1$ for all $i$ and $j$.
	We obtain the following extension of Theorem~\ref{thm:main_r}.
	
	\begin{theorem}\label{thm:main_r_abelian}
		For every positive integer $r$  
		there exists a function $\sign_r : \Par_r \to \{1,-1\}$ such that for every finite abelian group $G$ of order $r$, 
		every $r$-partite partition $\blambda=(\lambda^{(0)},\ldots,\lambda^{(r-1)})$ of a positive integer $n$,
		and every composition $\mu = (\mu_1,\ldots,\mu_t)$ of $n$: 
		\[
		\psi^{(\lambda^{(0)},\ldots,\lambda^{(r-1)})}_{(\mu,\varnothing,\ldots,\varnothing)} = \sign_r(\lambda)\cdot \chi^\lambda_{r\mu},
		\]
		where 
		$\psi^{(\lambda^{(0)},\ldots,\lambda^{(r-1)})}$ is the irreducible $G \wr S_n$-character indexed by $(\lambda^{(0)},\ldots,\lambda^{(r-1)})$,
		$\chi^\lambda$ is the irreducible $S_{rn}$-character indexed by $\lambda := \varphi_r(\lambda^{(0)}, \ldots, \lambda^{(r-1)}) \in \Par_r$,
		and $r\mu := (r\mu_1,\dots,r\mu_t)$.
	\end{theorem}

	\begin{proof}
		The proof is exactly the same as for the special case $G = \ZZ_r$, since 
		the combinatorics of peelings described in Subsection~\ref{sec:peelings} 
		(leading to Theorem~\ref{thm:MN_boundary}) 
		and analyzed in Section~\ref{sec:main} 
		(leading to Theorem~\ref{thm:main_r}) 
		is the same. 
	\end{proof}
	
	\begin{remark}
		The function $\sign_r$ depends on the size $r = |G|$ only, and not on the structure of $G$.
	\end{remark}

	
	An alternative algebraic proof of Theorem~\ref{thm:main_r_abelian} was suggested by the anonymous referee. 
	
	\section{Alternative descriptions}%
	\label{sec:additional}
	
	We conclude with some alternative descriptions of $\sign_r(\lambda)$.
	
	%
	%
	
	\begin{definition}\label{def.distance}
		For a partition $\lambda$ with an empty $r$-core, 
		let $k$ be any integer larger or equal to the number of parts in $\lambda$.
		Define $d_r(\lambda,\varnothing)$ to be the minimal number of adjacent transpositions needed to transform the word $a^{(k)}(\lambda) \in [0,r-1]^k$ into the word $a^{(k)}(\varnothing) \in [0,r-1]^k$, where $\varnothing$ is the empty partition represented by a sequence of $k$ zeros.
		Note that $d_r(\lambda,\varnothing)$ is independent of the choice of $k$.
	\end{definition} 
	
	\begin{example}\label{ex:distance}
		For $\lambda$ and $k$ as in Example~\ref{ex:color_sequence},
		$a^{(k)}(\lambda) = (0,2,0,1,1)$
		and $a^{(k)}(\varnothing) = (1,0,2,1,0)$,
		so that $d_r(\lambda,\varnothing) = 4$.
	\end{example}
	
	
	Observation~\ref{t:count_transpositions} implies
	
	\begin{observation}\label{lem:sign2}
		The $\sign_r$ function in Theorem~\ref{thm:main_r} satisfies
		\[
		\sign_r(\lambda) = (-1)^{d_r(\lambda,\varnothing)}.
		\]
	\end{observation}
	
	
	
	For $r=2$ there is also a surprisingly simple formula, observed in~\cite{LPA}.
	
	\begin{corollary}\cite[Prop. 5.4]{LPA}\label{cor:Ayyer}
		For every partition $\lambda\vdash 2n$ with an empty $2$-core  
		\[
		\sign_2(\lambda)=(-1)^{\odd(\lambda)/2},
		\]
		where $\odd(\lambda)$ is the number of odd parts in $\lambda$.
	\end{corollary}
	
	\begin{proof}
		By induction on $n$. 
		If $n=0$, then $\sign_2(\varnothing)=1$,  $\odd(\varnothing)=0$ and equality holds. 
		For evey $\lambda\in \Par_2(2n)$ there 
		exists a ribbon $\nu$ of size $2$ (``domino'') 
		such that $\lambda\setminus\nu\in \Par_2(2n-2)$.
		If $\nu$ is horizontal then $\odd(\lambda) = \odd(\lambda\setminus\nu)$ and $a(\lambda\setminus\nu)=a(\lambda)$, thus by Observation~\ref{lem:sign2},
		$\sign_2(\lambda)=\sign_2(\lambda\setminus\nu)$.
		If $\nu$ is vertical then $\odd(\lambda)=\odd(\lambda\setminus\nu)\pm 2$ and $a(\lambda)$ is obtained from 
		$a(\lambda\setminus\nu)$ by switching two adjacent entries, hence  by Observation~\ref{lem:sign2},
		$\sign_2(\lambda)=-\sign_2(\lambda\setminus\nu)$. The induction hypothesis completes the proof in both cases. 
	\end{proof}

	\begin{question}
		Is there a similar formula for other values of $r$?
	\end{question}

	\medskip
	
	\noindent {\bf Acknowledgements.} The authors thank the anonymous referees for 
	helpful comments. 
	In particular, in the original version of this paper the main result was stated for a finite cyclic group $G$; the observation that the result holds, with the same proof, for any finite abelian group is due to a referee's comment.

\end{document}